\documentclass[11pt]{article}

\usepackage{newpxtext}
\usepackage[euler-digits]{eulervm}

\usepackage{amsthm} 
\usepackage{amsfonts}

\usepackage{bm}
\usepackage[english]{babel}% for Brittish English
\usepackage{microtype}		% for character protrusion and font expansion (only with pdflatex)
\usepackage{booktabs,ctable,multirow}
\usepackage{graphicx}
\usepackage{float}
\usepackage{array}
\usepackage{amsmath}
\usepackage{url} % not crucial - just used below for the URL 

\usepackage{xcolor}
\definecolor{violet}{RGB}{12,6,120}
\usepackage[bookmarks,bookmarksopen,bookmarksnumbered, colorlinks=true,citecolor=violet]{hyperref}

\usepackage{xspace}
\usepackage{xparse}             %for new document command

\newcommand{\eqdef}{\stackrel{\text{\tiny def}}{=}}
\newtheorem{corollary}{Corollary}
\newtheorem{definition}{Definition}

\newtheorem{lemma}{Lemma}

\newtheorem{theorem}{Theorem}

\DeclareRobustCommand{\argmin}[1]{\underset{#1}{\theargmin}\mspace{4mu}}

\DeclareRobustCommand{\E}[1]{\mathbb{E} \left [#1 \right]}             % expectation

\DeclareRobustCommand{\fm}[1]{\bm{\mu} \mspace{-2mu} \big [\mspace{-2mu} #1 \mspace{-2mu}\big]}                 % frechet mean
\DeclareRobustCommand{\md}[1]{\bm{m} \mspace{-2mu}   \big [\mspace{-2mu} #1 \big]}    % median

\DeclareRobustCommand{\sfm}[1]{\widehat{\bm{\mu}}_N \mspace{-2mu}\big [\mspace{-2mu}#1 \mspace{-2mu}\big]}    % frechet mean 
\DeclareRobustCommand{\smd}[1]{\widehat{\bm{m}}_N\left [#1 \right]}                                           % median

\DeclareRobustCommand{\sE}[1]{\widehat{\mathbb{E}}_N\left [#1 \right]} % sample expectation

\DeclareRobustCommand{\prob}[1]{\mspace{2mu}\mathbb{P}\mspace{-2mu}\left(\mspace{-2mu} #1 \mspace{-2mu}\right)}

\DeclareRobustCommand{\var}[1]{\variance\left[#1\right]}

\DeclareMathOperator{\theargmin}{argmin}

\DeclareMathOperator{\variance}{var}
\DeclareRobustCommand{\o}[1]{{\scriptstyle \mathcal{O}} \left(#1\right)}

\DeclareMathOperator{\pr}{\mathbb{P}}

\NewDocumentCommand \F {o m}   %m = mandatory in {}, o = optional in []
{
  \IfNoValueTF {#2} {\widehat{F}\left(#2\right)}{\widehat{F}_{#1}\mspace{-4mu}\left(#2\right) }
}

\DeclareRobustCommand{\o}[1]{{\scriptstyle \mathcal{O}} \left(#1\right)}
\DeclareRobustCommand{\O}[1]{{\mathcal{O}} \mspace{-4mu} \left(#1\right)}

\newcommand{\iplus}{\mspace{-2mu}\left \lvert \cE (\bB) \right \rvert \mspace{-2mu}}

\newcommand{\bA}{\bm{A}}
\newcommand{\bAp}{\bm{A}^\prime}
\newcommand{\bAk}{\bm{A}^{(k)}}

\newcommand{\aij}{a_{ij}}
\newcommand{\aijp}{a_{i^\prime j^\prime}}

\newcommand{\akij}{a^{(k)}_{ij}}
\newcommand{\akijp}{a^{(k)}_{i^\prime j^\prime}}

\newcommand{\roij}{\rho_{ij, i^\prime j^\prime}}
\newcommand{\pij}{p_{ij}}
\newcommand{\pijp}{p_{i^\prime j^\prime}}

\newcommand{\bB}{\bm{B}}

\newcommand{\bP}{\bm{P}}

\newcommand{\cE}{\mathcal{E}}
\newcommand{\nE}{\overline{\mathcal{E}}}

\newcommand{\cG}{\mathcal{G}}

\newcommand{\cS}{\mathcal{S}}

\newcommand{\gnP}{\mathcal{G}\mspace{-1mu}\big(\mspace{-2mu}n,\bP\mspace{-1mu}\big)}
\newcommand{\Gk}{G^{(k)}}

\newcommand{\ER}{Erd\H{o}s-R\'enyi\xspace}
\newcommand{\fr}{Fr\'echet\xspace}

\newcommand{\citeg}[1]{(e.g., \cite{#1})}

\begin{document}
\title{Sharp Threshold for the Fr\'echet Mean (or Median) of Inhomogeneous Erd\H{o}s-R\'enyi Random Graphs}
\author{Fran\c{c}ois G. Meyer\thanks{F.G.M was supported by the National Natural Science Foundation
    (\href{https://www.nsf.gov/awardsearch/showAward?AWD_ID=1815971}{CCF/CIF 1815971}).
  }\\
  Applied Mathematics, University of Colorado at Boulder, Boulder CO 80305\\
  \href{mailto:fmeyer@colorado.edu}{\sf \small fmeyer@colorado.edu}\\{\small \url{https://francoismeyer.github.io}}
}
\date{}
\maketitle
%_______________________________________________________________________
\begin{abstract}
  We address the following foundational question: what is the population, and sample, Frechet mean (or median) graph of
  an ensemble of inhomogeneous \ER random graphs? We prove that if we use the Hamming distance to compute distances
  between graphs, then the Frechet mean (or median) graph of the inhomogeneous random graphs with expected adjacency
  matrix $\mathbf{P}$ is obtained by thresholding $\mathbf{P}$: an edge exists between the vertices $i$ and $j$ in the
  mean graph if and only if $p_{ij} > 1/2$.  We show that the result also holds for the sample mean (or median) when
  $\mathbf{P}$ is replaced with the sample mean adjacency matrix. Consequently, the \fr mean (or median) graph of
  inhomogeneous \ER random graphs exhibits a sharp threshold: it is either the empty graph if $\max p_{ij} < 1/2$, or
  the complete graph if $ \min p_{ij} > 1/2 $. This novel theoretical result has some significant practical
  consequences; for instance, the \fr mean of an ensemble of sparse inhomogeneous random graphs is always the empty graph.
\end{abstract}
\noindent%
{\it Keywords:}  \fr mean; \fr median; statistical network analysis.
%________________________________________________________________________
\section{Introduction}
% ________________________________________________________________________
The \fr mean (or median) graph, which extends the notion of mean to probability measures defined on metric spaces
\cite{frechet47}, has become a standard tool for the analysis of graph-valued data
\citeg{dubey20,ferrer10,ginestet17,jain16b,kolaczyk20,lunagomez20}. At the same time, inhomogeneous \ER random graphs
\cite{kovalenko71,bollobas07} have great practical importance since they provide tractable models that capture many of
the topological structures of real networks \cite{goldenberg10,kolaczyk17}. 

Let $\cG$ be the set of all simple labeled graphs with vertex set $\left\{1, \ldots ,n\right\}$, and let $\cS$ be the set of
$n \times n$ adjacency matrices of graphs in $\cG$,
\begin{equation}
  \cS = \left \{
    \bA \in \{0,1\}^{n \times n}; \text{where} \; a_{ij} = a_{ji},\text{and}  \; a_{i,i} = 0; \; 1 \leq i < j \leq n
  \right\}.
    \label{adjacency_matrices}
\end{equation}
We denote by $\gnP$, the probability space formed by inhomogeneous \ER random graphs \cite{kovalenko71,bollobas07}, which is
defined as follows. We assign to every graph $G \in \cG$, with adjacency matrix $\bA$, the probability 
\begin{equation}
  \prob{\bA} = \prod_{1 \le 1 < j \le n} \left [p_{ij} \right]^{a_{ij}} \left [1- p_{ij} \right]^{1 - a_{ij}}.
  \label{laproba}
\end{equation}
The $n \times n$ matrix $\bP = \left [p_{ij}\right]$ determines the edge probabilities $0 \le p_{ij} \le 1$, with $p_{ii} = 0$.
Throughout the paper, we identify $\gnP$ with the probability space $(\cS, \pr)$, where $\pr$ is defined by (\ref{laproba}).

The prominence of $\gnP$ stems from its ability to provide tractable models of random graphs that can capture many of
the structures of real networks (e.g., stochastic block models \cite{abbe17,fortunato16} to model communities
\cite{snijders11}). We equip $\cG$ with the Hamming distance defined as follows.
% ________________________________________________________________________
\begin{definition}
The Hamming distance between $G$ and  $G^\prime$ in $\cG$, with adjacency matrix $\bA$ and $\bAp$ respectively, is given by
  \begin{equation}
    d_H(G,G^\prime) = \sum_{1 \leq i< j \leq n} \lvert a_{ij} - a^\prime_{ij}\rvert.
  \end{equation}
    \label{hamming}
\end{definition}
% ________________________________________________________________________
We characterize the mean of the probability $\pr$ with the \fr mean and median graphs,
\cite{frechet47,schweizer60}, that are defined as follows.
% ________________________________________________________________________
\begin{definition}
  The \fr mean of the probability measure $\pr$ is the set formed by the solutions to
  \begin{equation}
    \fm{\pr} =  \argmin{G\in \cG} \sum_{G^\prime \in \cG} d_H^2(G,G^\prime) \prob{G^\prime}
    \label{mean}
  \end{equation}
  and the \fr median of the probability measure $\pr$ is the set formed by the solutions to
  \begin{equation}
    \md{\pr} = \argmin{G\in \cG} \sum_{H \in \cG} d_H(G,H) \prob{H}
    \label{median},
  \end{equation}
where $d_H$ is the Hamming distance (\ref{hamming}).
\end{definition}
% ________________________________________________________________________
We note that solutions to the minimization problems (\ref{mean}) and (\ref{median}) 
always exist, but need not be unique. Because all the results in this paper hold for any graph in the set formed by the
solutions to (\ref{mean}) and (\ref{median}), and without any loss of generality, we assume that $\fm{\pr}$
and $\sfm{\pr}$ each contains a single element.

Because the focus of this work is not the computation of the \fr mean graph, but rather a theoretical analysis of the properties
that the \fr mean graph inherits from the probability measure $\pr$, defined in (\ref{laproba}), we can assume that all the
graphs are defined on the same vertex set.\\

This notion of centrality is well adapted to metric spaces (since graph sets are not Euclidean spaces
\citeg{chowdhury18,jain12,jain16a,jain16b,kolaczyk20}.  The vital role played by the \fr mean as a location parameter, is
exemplified in the work of \cite{banks98} and \cite{lunagomez20}, who have created novel families of random graphs by generating
random perturbations around a given \fr mean (also called modal or central graph in \cite{banks98}). In practice, the \fr mean
itself is computed from a training set of graphs that display specific topological features of interest. To take full advantage
of the training set, one needs to insure that the sample \fr mean inherits from the training set the desired topological
structure.\\

By replacing $\pr$ with the empirical measure, the concept of \fr mean and median graphs can be extended to a sample of
graphs defined on the same vertex set $\left\{1, \ldots ,n\right\}$.
% ________________________________________________________________________
\begin{definition}
  Let $\left\{ \Gk \right\}_{1\leq k \leq N}$,be independent random graphs, sampled from $\pr$. The sample \fr mean is the set
  composed of the solutions to
  \begin{equation}
    \sfm{\pr} =  \argmin{G\in \cG} \frac{1}{N}\sum_{k=1}^N d^2(G,\Gk);
    \label{sample-frechet-mean}
  \end{equation}
the sample \fr median is the set composed of the solutions to
  \begin{equation}
    \smd{\pr} =  \argmin{G\in \cG} \frac{1}{N}\sum_{k=1}^N d(G,\Gk).
    \label{sample-frechet-median}
  \end{equation}
\end{definition}
% ________________________________________________________________________
Several algorithms have been proposed to compute the sample \fr mean and median when the distance $d$ is the edit distance
(e.g., \cite{bardaji10,ferrer09,jiang01}), or the Euclidean distance (e.g., \cite{jain08,jain09}).

% ________________________________________________________________________
\subsection{Our main contributions}
% ________________________________________________________________________
The prominence of the inhomogeneous \ER random graph model \cite{bollobas07} prompts the following critical question: does the
\fr mean of $\pr$ inherit from the probability space $\gnP$ any of the edge connectivity information encoded by
$\bP$?

In this paper, we answer this question. We show in Theorem~\ref{theorem1} that the population \fr mean graph $\fm{\pr}$ can be
obtained by thresholding the mean adjacency matrix $\E{\bA}=\bP$; an edge exists between the vertices $i$ and $j$ in $\fm{\bA}$ if
and only if $\E{\bA}_{ij} > 1/2$.  We prove in Theorem~\ref{theorem2} that this result also holds for the sample \fr mean graph,
$\sfm{\pr}$, when $\E{\bA}$ is replaced with the sample mean adjacency matrix, $\sE{\bA}$.

%________________________________________________________________________
\section{Main Results
  \label{results}}

Let $\bP = \left [p_{ij}\right]$ be an $n \times n$ symmetric matrix with entries $0 \le p_{ij} \le 1$. In the following
two theorems we evaluate the \fr mean (or median) graph, and the sample mean (or median) graph of $\gnP$. We first
consider the \fr mean graph and median graph. We denote by $[n]$ the set $\{1,\ldots,n\}$.
%________________________________________________________________________
\subsection{The population \fr mean graph and median graph of $\mathcal{G}(n,\mathbb{P})$}
%________________________________________________________________________
\begin{theorem}
  \label{theorem1}
  Let $\md{\bA}$ and $\fm{\bA}$ be the adjacency matrices of the \fr median graph $\md{\pr}$ and mean graph $\fm{\pr}$
  respectively. Then $\md{\bA}$ and $\fm{\bA}$ are given by
  \begin{equation}
    \forall i,j \in [n],\quad
            \Big[\md{\bA}\Big]_{ij} = \Big[\fm{\bA}\Big]_{ij} = 
    \begin{cases}
      1 & \text{if} \quad p_{ij} > 1/2,\\
      0 & \text{otherwise.}
    \end{cases}
    \label{populationFrechetMean}
  \end{equation}
\end{theorem}
% ________________________________________________________________________
\begin{proof}
  The proof is given in Section \ref{proofTheorem1}.
\end{proof}
% ________________________________________________________________________
\subsection{The sample \fr mean graph of a graph sample in $\mathcal{G}(n,\mathbb{P})$}
% ________________________________________________________________________
We now turn our attention to the sample \fr mean graph. The computation of the sample \fr mean graph using the Hamming distance
is NP-hard~\citeg{chen19}. For this reason, several alternatives have been proposed \citeg{ferrer10,ginestet17}.

Before presenting the second result, we take a short detour through the sample \fr median graph
(e.g., \cite{han16,jiang01,mukherjee09}), minimiser of (\ref{sample-frechet-median}), and which can be computed using the
majority rule \cite{banks98}.
% ________________________________________________________________________
\begin{lemma}[\cite{banks98}]
  \label{lemma8}
  The adjacency matrix $\smd{\bA}$ of the sample \fr mean graph $\smd{\pr}$ is given by
  \begin{equation}
    \forall i,j \in [n],\quad    \Big[ \smd{\bA}_{ij} \Big] =
    \begin{cases}
      1 & \text{if} \; \sum_{k=1}^N \akij \ge N/2,\\
      0 &  \text{otherwise.}
    \end{cases}
    \label{majority-rule}
  \end{equation}
\end{lemma}
% ________________________________________________________________________
We now come back to the second main contribution, where we prove that the sample \fr mean graph of $N$ independent random graphs
from $\gnP$ is asymptotically equal (for large sample size $N$) to the sample \fr median graph, with high probability.
% ________________________________________________________________________
\begin{theorem}
  \label{theorem2}
  $\forall \delta \in (0,1), \exists N_\delta, \forall N \ge N_\delta$,  $\smd{\bA}$ and $\sfm{\bA}$ are given by
  \begin{equation}
    \forall i,j \in \left[n\right], \quad 
    \Big[ \sfm{\bA} \Big]_{ij}  =
    \Big[ \smd{\bA} \Big]_{ij}  =
    \begin{cases}
      1 & \text{if} \quad  \E{\bA}_{ij} =\pij > 1/2,\\
      0 &  \text{otherwise,}
    \end{cases}
    \label{sampleFrechetMean}
  \end{equation}
  with probability $ 1- \delta$ over the realizations of the graphs $\left\{G^{(1)}, \ldots, G^{(N)}\right\}$ in $\gnP$. 
\end{theorem}
% ________________________________________________________________________
\begin{proof}
  The proof is given in section \ref{proofTheorem2}.
\end{proof}
% ________________________________________________________________________
The practical impact of Theorem~\ref{theorem2} is given by the following corollary, which is an elementary consequence of
theorem~\ref{theorem2} and lemma~\ref{lemma8}.
% ________________________________________________________________________
\begin{corollary}
  \label{corollary1}
  $\forall \delta \in (0,1), \exists N_\delta, \forall N\ge N_\delta$, the sample \fr mean graph $\sfm{\pr}$ is given by the majority rule
  \begin{equation}
    \forall i,j \in \left[n\right], \quad 
    \sfm{\bA}_{ij} =
    \begin{cases}
      1 & \text{if} \quad \sum_{k=1}^N \akij > N/2,\\
      0 &  \text{otherwise,}
    \end{cases}
    \label{practicalFrechetMean}
  \end{equation}
  with probability $ 1- \delta$ over the realizations of the graphs, $\left\{G^{(1)}, \ldots, G^{(N)}\right\}$, in $\gnP$.
\end{corollary}
This novel theoretical result has some significant practical consequences; consider for instance sparse graphs where
\mbox{$\min p_{ij} < 1/2$} (e.g., graphs with $\o{n^2}$ but $\omega(n)$ edges), then the sample \fr mean is the empty graph, and
is pointless. More generally, the sample \fr mean computed from a training set of graphs, which display specific topological
features of interest, will not inherit from the training set the desired topological structure.
%________________________________________________________________________
\section{Simulation Studies
  \label{experiments}} We compare our theoretical analysis to finite sample estimates that were computed using numerical
simulations. All graphs were generated using the $\gnP$ model (\ref{laproba}). The sample \fr mean was computed using
(\ref{practicalFrechetMean}). All graphs had $n=512$ vertices. The sample size was $N =1,000$. The software used to
conduct the experiments is publicly available \cite{meyer22b}. We explored the switching of the mean graph from the
empty graph to the complete graph, as the entries in the probability matrix $\bP$ shift from being all less than $1/2$
to being all larger than $1/2$.

We investigate this shift over a large set of matrices $\bP$ that are generated at random. During the simulation, each
$\bP$ is created by populating its entries at random using independent (up to symmetry) beta random variables,
\begin{equation}
p_{ij} \sim \text{beta} (\nu,\omega),
\end{equation}
and thus the edge probability $p_{ij}$ becomes a random variable with mean and variance given by,
\begin{equation}
\E{p_{ij}} = \frac{\nu}{(\nu+\omega)}, \quad \text{and} \quad \var{p_{ij}} =
\frac{\nu\omega}{[(\nu+\omega)^2(\nu+\omega+1)]}.
\end{equation}
In our experiments, we fix $\nu + \omega=64$, and we let $\nu$ vary over the interval $(0,\nu+\omega)$, thereby
exploring the range $[0,1]$ for $\E{p_{ij}}$. This process makes it possible to probe the behaviour of $\sfm{\pr}$ over
a large range of possible edge connectivity: from the very sparse to the very dense graphs.\\

We denote by $\lvert \cE(\sfm{\pr}) \big \rvert$ the number of edges of the sample \fr mean graph
$\sfm{\pr}$. Fig.~\ref{fig3} displays $\lvert \cE(\sfm{\pr}) \big \rvert$ as a function of the mean edge density,
$\E{p_{ij}}$. We normalize $\big \lvert \cE(\sfm{\pr}) \big \rvert$ by dividing the number edges of $\sfm{\pr}$ by the
maximum possible number of edges,
\begin{equation}
\frac{2\big \lvert \cE(\sfm{\pr}) \big \rvert}{n(n+1)}.
\end{equation}
For each value of $\E{p_{ij}}$ on the $x$-axis in Fig.~\ref{fig3}, we generate 16 independent random realizations of the
$N=1,000$ sample random graphs in $\gnP$. We compute the average number of edges of $\sfm{\pr}$ over the 16 realizations,
and report this number on the $y$-axis of Fig.~\ref{fig3}.  Each curve corresponds to a different value of $\nu$. We
increase $\nu$ from $8$ to $64$ by $8$ each time. As $\nu$ increases, the variance on the distribution of the entries
$p_{ij}$ decreases, the curve in Fig.~\ref{fig3} becomes steeper, and the switch from the empty graph to the complete
graph becomes more sudden.%%
%________________________________________________________________________
\begin{figure}[H]
  \centerline{
    \includegraphics[width=0.45\textwidth]{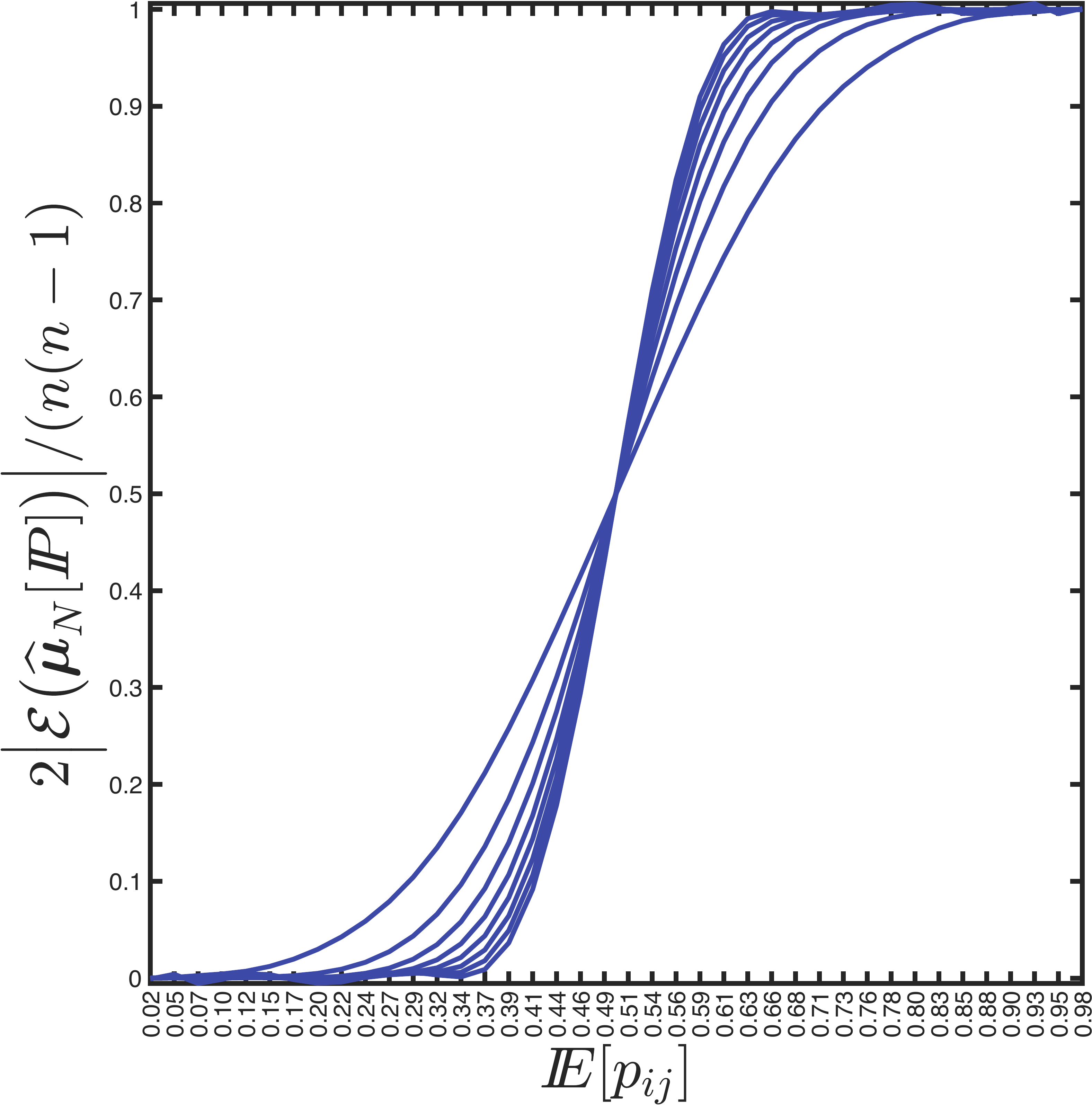}
  }
  \caption{Fraction  of the number of edges in $\sfm{\pr}$ as a  function of the location parameter $\E{p_{ij}} = \nu/64$.
The eight curves correspond to $\nu = 8, 16, 24, 32, 40, 48, 64$.
    \label{fig3}}
\end{figure}
%_________________________________________________________________________
\noindent The experiment confirms that the \fr mean of inhomogeneous \ER random graphs exhibits a sharp threshold: it
rapidly switches from the empty graph to the complete graph as the expected edge probability $\E{p_{ij}}$ becomes larger
than $1/2$.%%

\section{Discussion and Conclusion}
% ________________________________________________________________________
In this work we derived the expression for the population \fr mean for inhomogeneous \ER random graphs. We proved that the
sample \fr mean was consistent, and could be estimated using a simple thresholding rule. Our results have several practical
implications.

First, our work implies that the sample \fr mean computed from a training set of graphs, which display specific topological
features of interest, will not inherit from the training set the desired topological structure. Indeed, in the context of
inhomogeneous \ER random graphs, the (population or sample) \fr mean graph no longer captures the
edge density encoded by the edge probability.\\

Our answer to the question of the authors in \cite{lunagomez20}: ``what is the ``mean'' network (rather than how do we estimate
the success-probabilities of an inhomogeneous random graph), and do we want the ``mean'' itself to be a network?''  is therefore
disappointing in the context of the probability space $\gnP$. While the \fr mean is indeed an element of $\gnP$, it only
provides a simplistic sketch of that probability space. Consider for instance sparse graphs where \mbox{$\min p_{ij} < 1/2$}
(e.g., graphs with $\o{n^2}$ but $\omega(n)$ edges), then the sample \fr mean is the empty graph, and is pointless.\\

On a more positive note, our analysis provides a theoretical justification for several algorithms designed to recover a graph
from noisy measurements of its adjacency matrix. For instance, the authors in \cite{josephs21} devise a method to recover a
fixed network from unlabeled noisy samples. Instead of estimating the \fr mean, they compute the sample mean of the noisy
adjacency matrices, and threshold the sample mean to recover an unweighted graph.

Our results offer a theoretical justification of the approach of \cite{josephs21}, if one assumes that the noisy graphs are
aligned and are realizations of inhomogeneous \ER random graphs, with an unknown edge probability matrix $\bP$, then the
algorithm described in \cite{josephs21} recovers the sample \fr mean graph.

% ________________________________________________________________________
%________________________________________________________________________
\section{Proofs of the main results}
% ________________________________________________________________________
We give in the following the proofs of theorems 1 and 2. In the process, we prove several technical lemmata.
% ________________________________________________________________________
\subsection{The \fr functions for the Hamming distance}
% ________________________________________________________________________
We define the \fr functions, associated with the \fr median and mean, and  the sample \fr mean and median.
% ________________________________________________________________________
\begin{definition}
  We denote by $F_q$ the \fr function associated with the \fr median ($q=1$) or mean ($q=2$),
  \begin{equation}
    F_q(\bB)  =  \sum_{\bA \in \cS} d^q_H(\bA, \bB) \prob{\bA}.
  \end{equation}
  We denote by $\widehat{F}_q$ the sample \fr function associated with the sample \fr median ($q=1$) or mean ($q=2$), 
  \begin{equation}
    \F[q]{\bB}  =  \frac{1}{N} \sum_{k=1}^N d^q_H(\bAk, \bB).
  \end{equation}
\end{definition}
% ________________________________________________________________________
Let $\bA$ and $\bB$ be two adjacency matrices in $\cS$. We derive in the following lemma an expression for the Hamming distance
squared, $d^2_H(\bA, \bB)$ where the computation is split between the entries of $\bA$ along the edges of $\bB$, $\cE(\bB)$, and
the entries of $\bA$ along the ``nonedges'' of $\bB$, $\nE(\bB)$. We denote by $\iplus$ the number of edges in $\bB$.
% ________________________________________________________________________
\begin{lemma}
  \label{F2ofBlemma}  
  Let $\bA$ and $\bB$ two matrices in $\cS$. Then,
  \begin{align}
    d^2_H(\bA,\bB) = & \bigg[\sum_{1 \leq i < j \leq n}\mspace{-16mu} \aij \bigg]^2 \mspace{-8mu} + \iplus^2
                       + 2 \iplus
                       \bigg[
                       \sum_{(i,j) \in \nE(\bB)}\mspace{-16mu} \aij  - \mspace{-16mu}
                       \sum_{(i,j) \in \cE(\bB)} \mspace{-16mu} \aij
                       \bigg] \notag\\
                     & - 4  \mspace{-8mu} \sum_{(i,j) \in \cE(\bB)}\mspace{-2mu} \sum_{(i^\prime,j^\prime) \in \nE(\bB)}
                       \mspace{-16mu} \aij \aijp \label{F2ofB}
  \end{align}
\end{lemma}
% ________________________________________________________________________
The proof of (\ref{F2ofB}) is elementary, and is omitted. As explained in the next lemma, the values $F_1(\bB)$ and
$F_2(\bB)$ depend only on the entries of the probability matrix $\bP$ along edges of $\bB$.
% ________________________________________________________________________
\begin{lemma}
  \label{lemma1}
  Let $\bB \in \cS$, let $\cE(\bB)$ be the set of edges of the graph associated to $\bB$. Then
  \begin{equation}
    F_1(\bB) =
    \sum_{(i,j) \in \cE(\bB)} \mspace{-12mu} (1-2 p_{ij}) 
    + \mspace{-8mu}\sum_{1 \le i < j \le n} \mspace{-12mu} p_{ij}.
    \label{frechetMedianOpt}
  \end{equation}
  \begin{equation}
    F_2(\bB) = \left[
      \sum_{(i,j) \in \cE(\bB)} \left( 1 - 2 p_{ij} \right)
      + \mspace{-12mu} \sum_{1\le i < j \le n} \mspace{-8mu} p_{ij}
    \right]^2
    + \sum_{1 \leq i < j \leq n} p_{ij}(1 - p_{ij}).
    \label{frechetMeanOpt}
  \end{equation}
\end{lemma}
% ________________________________________________________________________
\begin{proof}
  We use lemma~\ref{F2ofBlemma}, and take the expectation with respect to the probability measure $\pr$, on both sides of
  equation (\ref{F2ofB})
\begin{align}
  \sum_{\bA \in \cS} d^2_H(\bA,\bB)\prob{\bA} =
  & \sum_{\bA \in \cS} \left[\sum_{1 \leq i < j \leq n}\aij \right]^2  \mspace{-12mu} \prob{\bA}
    + \iplus^2  \label{square}\\
  & - 4  \mspace{-8mu}
    \sum_{(i,j) \in \cE(\bB)}\sum_{(i^\prime,j^\prime) \in \nE(\bB)}
    \sum_{\bA \in \cS} \aij \aijp \prob{\bA} \label{cross1}\\
  & + 2 \iplus  \left[
    \sum_{(i,j) \in \nE(\bB)} \sum_{\bA \in \cS} \aij \prob{\bA}
    - \sum_{(i,j) \in \cE(\bB)} \sum_{\bA \in \cS} \aij \prob{\bA}
    \right] \label{cross2}.
\end{align}
Now, the expectation of $\aij$ is given by
\begin{equation}
  \E{\aij} = \sum_{\bA \in \cS} \aij \prob{\bA} = p_{ij}, \label{expectaij}
\end{equation}
and because the edges $(i,j) \in \cE(\bB)$ and $(i^\prime,j^\prime) \in \nE(\bB)$ are independent,
\begin{equation}
  \E{\aij \aijp}  = \sum_{\bA \in \cS} \aij \aijp \prob{\bA} = p_{ij} p_{i^\prime j^\prime}.
\end{equation}
Therefore (\ref{cross1}) becomes
\begin{equation}
  - 4  \mspace{-8mu}\sum_{(i,j) \in \cE(\bB)}\sum_{(i^\prime,j^\prime) \in \nE(\bB)}
  \sum_{\bA \in \cS} \aij \aijp \prob{\bA}
  =  - 4  
      \left[\sum_{(i,j) \in \cE(\bB)}p_{ij} \right]
      \left[\sum_{(i^\prime,j^\prime) \in \nE(\bB)} p_{i^\prime j^\prime} \right]
      \label{cross1b}
    \end{equation}
Using (\ref{expectaij}) one more time, (\ref{cross2}) is equal to,
\begin{align}
  2 \iplus  & \left[\sum_{(i,j) \in \nE(\bB)} \sum_{\bA \in \cS} \aij \prob{\bA}
              - \sum_{(i,j) \in \cE(\bB)} \sum_{\bA \in \cS} \aij \prob{\bA}\right]\\
  =  &2 \iplus 
       \left[\sum_{(i,j) \in \nE(\bB)} p_{ij} \right]
       - 2 \iplus
       \left[\sum_{(i,j) \in \cE(\bB)} p_{ij}\right] \label{cross2b}
\end{align}
The first term  in the righthand side of (\ref{square}) can be evaluated to give
\begin{equation}
  \sum_{\bA \in \cS} \left[\sum_{1 \leq i < j \leq n}\aij \right]^2 \prob{\bA}
  =  \sum_{1 \leq i < j \leq n} \mspace{4mu}
      \sum_{1 \leq i^\prime < j^\prime \leq n} \mspace{4mu}
      \sum_{\bA \in \cS} \aij \; \aijp \prob{\bA},
    \end{equation}
    where
\begin{equation}
  \sum_{\bA \in \cS} \aij \aijp \prob{\bA} =
  \E{\aij \aijp} = 
  \begin{cases}
    \E{\aij}\E{\aijp} = p^2_{ij} & \text{if $(i,j) \neq (i^\prime,j^\prime)$,}\\
    \E{\aij^2} = \E{\aij} = p_{ij}& \text{if $(i,j) = (i^\prime,j^\prime)$.}
  \end{cases}
\end{equation}
We conclude that
\begin{equation}
  \sum_{\bA \in \cS} \left[\sum_{1 \leq i < j \leq n}\aij \right]^2 \mspace{-12mu}\prob{\bA}
  = \mspace{-12mu}  \sum_{1 \leq i < j \leq n} \mspace{4mu}
  \sum_{\stackrel{1 \leq i^\prime < j^\prime \leq n}{(i,j) \neq (i^\prime,j^\prime)}} \mspace{-16mu}
  p^2_{ij} + \mspace{-8mu} \sum_{1 \leq i < j \leq n} \mspace{-8mu} p_{ij}
  =  \left[\sum_{1 \leq i < j \leq n} \mspace{-8mu} p_{ij}\right]^2
  \mspace{-8mu} + \mspace{-12mu} 
  \sum_{1 \leq i < j \leq n} \mspace{-8mu} p_{ij}(1 - p_{ij}). \label{squareb}
\end{equation}
We can substitute (\ref{cross1b}), (\ref{cross2b}), and (\ref{squareb}) into (\ref{cross1}), (\ref{cross2}), and
(\ref{square}) respectively, and we get
\begin{align}
  \sum_{\bA \in \cS} d^2_H(\bA,\bB) \prob{\bA}
  = & \left[\sum_{1 \leq i < j \leq n} p_{ij}\right]^2 
      + 2 \left[ \sum_{1 \leq i < j \leq n} p_{ij}\right]
      \left[\sum_{(i,j) \in \cE(\bB)} \mspace{-12mu} (1 -2p_{ij}) \right] \notag \\
  + & \left[\sum_{(i,j) \in \cE(\bB)} \mspace{-12mu} (1 -2p_{ij})\right]^2
      +   \sum_{1 \leq i < j \leq n} p_{ij}(1 - p_{ij}) \notag \\
  = & \left[\sum_{(i,j) \in \cE(\bB)} \mspace{-12mu} (1 -2p_{ij})
      + \sum_{1 \leq i < j \leq n} \mspace{-12mu} p_{ij}\right]^2
      + \sum_{1 \leq i < j \leq n} p_{ij}(1 - p_{ij}),
\end{align}
which matches the expression (\ref{frechetMeanOpt}).  
\end{proof}
% ________________________________________________________________________  
\subsection{The Fr\'echet mean and median of $\mathcal{G}\big(n,\mathbf{P}\big)$: proof of Theorem~\ref{theorem1}
  \label{proofTheorem1}}
% ________________________________________________________________________  
We prove theorem~\ref{theorem1} for the mean. The proof for the median is very similar; it is in fact simpler and is therefore
omitted. By lemma~\ref{lemma1}, we seek the matrix $\bB$, with edge set $\cE(\bB)$, that minimizes the \fr function defined by
(\ref{frechetMeanOpt}). Let us denote
  \begin{equation}
    x \eqdef \mspace{-12mu} \sum_{(i,j) \in \cE(\bB)} \left( 1 - 2 p_{ij} \right). \label{whatisx}
  \end{equation}
The variable $x$ is clearly a function of $\bB$, and $\bP$. To avoid unnecessary complicated notations, we keep these
dependencies implicit in our exposition. Since $0 \le p_{ij} \le 1$, $x$ is confined to the following interval,
  \begin{equation}
    - \mspace{-12mu} \sum_{1 \le i < j \le n} p_{ij} 
    \le
    - \mspace{-12mu} \sum_{(i,j) \in \cE(\bB)}  p_{ij}
    \le
    x
    \le 
    \sum_{(i,j) \in \cE(\bB)} \mspace{-16mu} 1
    \mspace{8mu} \le n(n-1)/2.
  \end{equation} 
  In fact, $x = -\sum_{1 \le i < j \le n} p_{ij}$, only if $\forall i,j \in [n], \; \pij =1$, and the graph associated
  to $\bB$ in (\ref{whatisx}) is the complete graph. This case is of no interest to us, and thus we can assume that
  $\bP$ is always chosen such that
\begin{equation}
  -\mspace{-16mu}\sum_{1 \le i < j \le n} p_{ij} < x.
  \label{leftBorder}
\end{equation}
We define,
\begin{equation*}
f(x) \eqdef \left[x + \sum_{1\le i < j \le n} p_{ij} \right]^2,
\end{equation*}
and we have
\begin{equation*}
f(x) = F_2(\bB) - \sum_{1 \leq i < j \leq n} p_{ij}(1 - p_{ij}).
\end{equation*}
Minimizing $F_2$ is therefore equivalent to minimizing $f$. Clearly, $f(x)$ is convex, has a global minimum at
$x_{\min{}} = -\sum_{1\le i< j \le n} p_{ij}$, and is increasing for $x \ge -\sum_{1\le i< j \le n} p_{ij}$. We seek $x^*$ that
minimizes $f(x)$ over the interval wherein $x$ is enclosed,
\begin{equation*}
\bigg (-\sum_{1\le i< j \le n} p_{ij}, \; n(n-1)/2 \bigg].
\end{equation*}
We note that because of (\ref{leftBorder}), $x_{\min{}} < x^\ast$. Also, $x^*$ cannot be positive; otherwise, we would get
$f(x^*) > f(0)$. The optimal value $x^*$ is obtained by minimizing the distance from $x^*$ to $-\sum_{1\le i< j \le n} p_{ij}$,
\begin{equation}
  x^\ast -  (- \mspace{-16mu}
  \sum_{1\le i< j \le n} p_{ij})
  =
  \mspace{-24mu}
  \sum_{(i,j) \in \cE(\bB)} \mspace{-16mu} (1 - 2 p_{ij})
  +
  \mspace{-16mu}
  \sum_{1\le i< j \le n} \mspace{-8mu} p_{ij}
  \ge
  \mspace{-24mu}
  \sum_{(i,j); 1 -2p_{ij} < 0} \mspace{-24mu} \left( 1 - 2 p_{ij} \right)
  +
  \mspace{-16mu}
  \sum_{1\le i< j \le n} p_{ij} \label{lowerbound}.
\end{equation}
The lower bound (\ref{lowerbound}) is independent of $\bB$, and can be obtained by choosing, 
\begin{equation}
  \fm{\bA}_{ij} =
  \begin{cases}
    1 & \text{if}\quad p_{ij} > 1/2,\\
    0 & \text{otherwise,}
  \end{cases}
\end{equation}
as advertised in the theorem. \qed
% ________________________________________________________________________
\subsection{The sample \fr functions for the Hamming distance}
We now consider $N$ independent random graphs, $\left\{ \Gk \right\}_{1\leq k \leq N}$, sampled from $\gnP$, with adjacency
matrices $\bAk$. The sample \fr functions associated with the sample \fr mean and median graph are given by the next lemma,
which corresponds to the sample version of lemma~\ref{lemma1}.
% ________________________________________________________________________
\begin{lemma}
  \label{lemma4}
  Let $\bB \in \cS$, let $\cE(\bB)$ be the set of edges of the graph associated to $\bB$. Then
  \begin{equation}
    \F [1]{\bB} =      \sum_{(i,j) \in \cE(\bB)} \left( 1 - 2 \sE{\aij} \right)
    + \mspace{-12mu} \sum_{1\le i < j \le n} \mspace{-8mu} \sE{\aij}
    \label{F1_sample}
  \end{equation}
  \begin{align}
    \F[2]{\bB}   =  & \left[
                      \sum_{(i,j) \in \cE(\bB)}  \mspace{-8mu} \left (1 - 2\sE{a_{ij}}\right) + \mspace{-8mu} \sum_{1 \le i < j \le n} \sE{a_{ij}}
                      \right]^2
                      \mspace{-12mu}  + \mspace{-12mu}
                      \sum_{1 \le i < j \le n} \sE{\aij} \left(1 - \sE{\aij} \right) \notag \\
                    & \mspace{16mu}
                      - \sum_{1 \le i < j \le n} \mspace{8mu}
                      \sum_{\stackrel{1 \le i^\prime < j^\prime \le n }{(i,j) \neq (i^\prime,j^\prime)}}
                      \left( \sE{\aij}  \sE{\aijp} - \sE{\roij}
                      \right) \notag\\
                    & +  4 \sum_{(i,j) \in \cE(\bB)} \sum_{(i^\prime,j^\prime) \in \nE(\bB)}
                      \left( \sE{\aij} \sE{\aijp} - \sE{\roij}
                      \right) \label{F2_sample} 
  \end{align}
  where the sample mean and sample correlation are defined by
  \begin{equation}
    \sE{a_{ij}} = \frac{1}{N} \sum_{k=1}^N \akij \quad \text{and} \quad \sE{\roij} = \frac{1}{N} \sum_{k=1}^N \akij \akijp
    \label{thesamples}
  \end{equation}
\end{lemma}    
% ________________________________________________________________________
\begin{proof}
  We prove (\ref{F2_sample}), the proof of (\ref{F1_sample}) is much simpler and is omitted.  The proof of (\ref{F2_sample}) is
  similar to the proof of lemma~\ref{lemma1}.  For each graph $\Gk$, we apply equation (\ref{F2ofB}), we sum over all the
  graphs in the sample, and divide by $N$ to get
\begin{align}
  \F[2]{\bB}
  = &  \iplus^2
      + 2  \iplus
      \left[\sum_{(i,j) \in \nE(\bB)} \frac{1}{N} \sum_{k=1}^N \akij
      -
      \sum_{(i,j) \in \cE(\bB)}  \frac{1}{N} \sum_{k=1}^N \akij
      \right]\\
  + &  \frac{1}{N} \sum_{k=1}^N \left[\sum_{1 \leq i < j \leq n}\akij \right]^2
      -  4  \sum_{(i,j) \in \nE(\bB)}\sum_{(i^\prime,j^\prime) \in \cE(\bB)} \left[\frac{1}{N} \sum_{k=1}^N \akij \akijp \right]\\
\end{align}
Using the expressions for the sample mean and correlation, in (\ref{thesamples}), we get 
\begin{align}
  \F[2]{\bB}
  = &  \iplus^2
      +
      2  \iplus \Big [\mspace{-16mu} \sum_{(i,j) \in \nE(\bB)} \sE{\aij}  - \sum_{(i,j) \in \cE(\bB)}  \sE{\aij} \Big] 
      + \frac{1}{N} \sum_{k=1}^N \Big [\mspace{-8mu} \sum_{1 \leq i < j \leq n}\akij \Big]^2 \notag \\
  -&  4  \sum_{(i,j) \in \nE(\bB)}\sum_{(i^\prime,j^\prime) \in \cE(\bB)} \mspace{-16mu} \sE{\roij} \label{Thenwehave}
\end{align}
We note that
\begin{equation}
  \frac{1}{N} \sum_{k=1}^N \Big[\mspace{-8mu} \sum_{1 \leq i < j \leq n}\akij \Big]^2
  =   \mspace{-12mu} \sum_{1 \le i < j \le n}\; \sum_{1 \le i^\prime <j^\prime \le n}  \frac{1}{N} \sum_{k=1}^N \akij \akijp
  =   \mspace{-12mu}  \sum_{1 \le i < j \le n}\; \sum_{1 \le i^\prime <j^\prime \le n}  \mspace{-16mu} \sE{\roij} \label{sommecarres}
\end{equation}
Also, we have
\begin{align}
  \iplus^2
  & + 2  \iplus   \bigg[\mspace{-4mu} \sum_{(i,j) \in \nE(\bB)} \sE{\aij}  - \sum_{(i,j) \in \cE(\bB)}  \sE{\aij} \bigg] \notag \\
  = & \bigg[ \iplus - 2 \mspace{-12mu} \sum_{(i,j) \in \cE(\bB)} \sE{\aij} \bigg]^2
      - \mspace{-12mu} \sum_{1 \le i < j \le n}\; \sum_{1 \le i^\prime <j^\prime \le n}  \mspace{-16mu} \sE{\aij}  \sE{\aijp} \notag \\
  & +   4   \mspace{-12mu}\sum_{(i,j) \in \nE(\bB)}\sum_{(i^\prime,j^\prime) \in \cE(\bB)}
    \mspace{-16mu} \sE{\aij}\sE{\aijp} \label{lecarre}
\end{align}
We can then substitute (\ref{sommecarres}) and (\ref{lecarre}) into (\ref{Thenwehave}), and we get
\begin{align}
  \F[2]{\bB}
  & = \bigg[ \iplus - 2 \mspace{-12mu} \sum_{(i,j) \in \cE(\bB)} \sE{\aij} \bigg]^2 
    \mspace{-8mu} - \mspace{-16mu} \sum_{1 \le i < j \le n}\; \sum_{1 \le i^\prime <j^\prime \le n}   
    \mspace{-8mu} \Big [\sE{\aij}  \sE{\aijp} - \sE{\roij} \Big]\notag\\
  & +   4  \mspace{-12mu} \sum_{(i,j) \in \nE(\bB)}\sum_{(i^\prime,j^\prime) \in \cE(\bB)}
    \mspace{-16mu} \Big [ \sE{\aij}\sE{\aijp}  - \sE{\roij}\Big] \label{weget}
\end{align}
Finally, we can extract from the second term in the first line of (\ref{weget}) the term that
corresponds to \mbox{$(i,j) = (i^\prime,j^\prime)$},
\begin{align}
  & \sum_{ 1\le i < j \le n}  \sum_{1 \le i^\prime <j^\prime \le n}
    \sE{\aij}  \sE{\aijp} - \sE{\roij}\notag \\
  & = \mspace{-8mu} \sum_{1 \le i < j \le n}  \sum_{\stackrel{1 \le i^\prime <j^\prime \le n}{(i^\prime,j^\prime)
    \neq  (i,j)}}
    \mspace{-8mu} \sE{\aij}  \sE{\aijp} - \sE{\roij}
    + \mspace{-16mu} \sum_{1 \le i < j \le n}
    \mspace{-16mu} \sE{a_{\aij}}  \sE{\aij} - \sE{\rho_{ij,ij}},
\end{align}
Now if $(i,j) = (i^\prime,j^\prime)$ we have
\begin{equation}
  \sE{\rho_{ij,ij}} = \sum_{k=1}^N \akij \akij = \sum_{k=1}^N \akij  = \sE{\aij},
\end{equation}
and therefore
\begin{align}
  &    \sum_{1 \le i < j \le n}  \sum_{1 \le i^\prime <j^\prime \le n}    
    \left[\sE{\aij}  \sE{\aijp} - \sE{\roij} \right] \notag\\
  & = \sum_{1 \le i < j \le n}\;
    \sum_{\stackrel{1 \le i^\prime <j^\prime \le n}{(i^\prime,j^\prime) \neq  (i,j)}}    
    \mspace{-12mu} \sE{\aij}  \sE{\aijp} - \sE{\roij}
    + \mspace{-12mu}  \sum_{1 \le i < j \le n}
    \mspace{-12mu} \sE{\aij} \left(\sE{\aij} -1 \right) \label{andso}
\end{align}
Substituting (\ref{andso}) into (\ref{weget}), we obtain  the result advertised in the lemma.
\end{proof}
% ________________________________________________________________________
\subsection{Concentration of the sample \fr functions for large sample size}
% ________________________________________________________________________
In the following two lemmata, we show that for large $N$, the sample \fr functions $\widehat{F}_1$ and $\widehat{F}_2$
concentrate around their population counterparts, $F_1$ and $F_2$.
% ________________________________________________________________________
\begin{lemma}
  \label{lemma5}
  For all $\delta \in (0,1)$, there exists $N_\delta$, such that for all $N\ge N_\delta$,
  \begin{equation}
    \F[1]{\bB} =
    \sum_{(i,j) \in \cE(\bB)}  \mspace{-8mu}
    \left (1 - 2\pij \right) + \mspace{-8mu} \sum_{1 \le i < j \le n} \pij
    + \O{\frac{1}{\sqrt{N}}},
  \end{equation}
  with probability $1-\delta$ over the realization of the sample $\left\{ \Gk \right\}_{1\leq k \leq N}$ .
\end{lemma}
% ________________________________________________________________________
\begin{proof}
% ________________________________________________________________________  
  The sample mean $\sE{\aij}$, defined in (\ref{thesamples}), is the sum of Bernoulli random variables, and it concentrates
  around its mean $\pij$. We use Hoeffding inequality to bound the variation of $\sE{\aij}$ around $\pij$. For each
  $ 1 \le i< j \le n$, we have,

\begin{equation}
  \prob{
    \mspace{-4mu} \bAk \sim \gnP;
    \left \lvert 
      \sE{\aij} -p_{ij} 
    \right \rvert
    \ge \varepsilon
  }
  \le
  \exp{\displaystyle \left(-2N \varepsilon^2\right)}.
\end{equation}
To control $\sum_{k=1}^N \akij$ for all $1 \le i < j < n$, we use a union bound, and we get,
\begin{equation}
  \forall 1 \le i < j < n,\quad
  \left \lvert \sE{\aij} -\pij \right \rvert \le \frac{\alpha}{\sqrt{N}},
  \label{meanConcentrates}
\end{equation}
with probability $1- \delta$, and where $  \alpha = \sqrt{\log(n/\sqrt{2\delta})}$.
From (\ref{meanConcentrates}), and for $N$ large enough we obtain the result announced in the lemma,
\begin{equation}
  \sum_{(i,j) \in \cE(\bB)}
  \mspace{-12mu}
  \left (1 - 2\sE{a_{ij}}\right)
  + \mspace{-16mu}
  \sum_{1 \le i < j \le n} \sE{a_{ij}}
  = \mspace{-24mu}
  \sum_{(i,j) \in \cE(\bB)}
  \mspace{-8mu}
  \left (1 - 2\pij\right)
  + \mspace{-16mu}
  \sum_{1 \le i < j \le n} 
  \mspace{-8mu}
  \pij
  + \O{\frac{1}{\sqrt{N}}}, 
\end{equation}
with probability $1-\delta$.
\end{proof}
% ________________________________________________________________________
Similarly, we show in the following lemma that for large sample size $N$, the sample \fr function $\F[2]{\bB}$ concentrates
around its population counterpart, $F_2(\bB)$.
\begin{lemma}
  \label{lemma6}
  For all $\delta \in (0,1)$, there exists $N_\delta$, such that for all $N\ge N_\delta$,
  \begin{equation}
    \F[2]{\bB} =
    \left[
      \sum_{(i,j) \in \cE(\bB)}  \mspace{-8mu}
      \left (1 - 2\pij \right) + \mspace{-8mu} \sum_{1 \le i < j \le n} \pij
    \right]^2
    \mspace{-12mu}  + \mspace{-12mu}
    \sum_{1 \le i < j \le n} \pij \left(1 - \pij \right)
    + \O{\frac{1}{\sqrt{N}}},
    \label{F2_concentrates}
  \end{equation}
  with probability $1-\delta$ over the realization of the sample $\left\{ \Gk \right\}_{1\leq k \leq N}$ .
\end{lemma}
% ________________________________________________________________________
\begin{proof}
  We have seen in the proof of lemma \ref{lemma5} that the sample mean concentrates around its population mean. From (\ref
  {meanConcentrates}) we have,
\begin{equation}
  \forall 1 \le i < j < n,\quad
  \left \lvert \sE{\aij} -\pij \right \rvert \le \frac{\alpha}{\sqrt{N}}.
\end{equation}
with probability $1- \delta/8$, and where $\alpha = \sqrt{\log{(2n/\sqrt{\delta}})}$. We now study the concentration of the
sample correlation,
\begin{equation}
  \sE{\roij} = \frac{1}{N} \sum_{k=1}^N \akij \akijp,
\end{equation}
when the pair of edges are distinct. Because $(i,j) \neq (i^\prime, j^\prime)$, the terms $\akij$ and $\akijp$ are always
independent, and the product $\akij \akijp$ is a Bernoulli random variable with parameter $\pij \pijp$. We conclude that the
sample correlation is the sum of Bernoullli random variables, and thus it concentrates around its mean, $\pij\pijp$.

  \begin{equation}
      \forall \; 1 \le i< j \le n, \mspace{4mu}
      \forall  \; 1 \le i^\prime< j^\prime\le n,\;
      \left \lvert \sE{\roij} -\pij \pijp \right \rvert \le \frac{\beta}{\sqrt{N}}
    \label{hoeffding3}
  \end{equation}
  with probability $1 - \delta/8$, where $\beta = \sqrt{\log{(n^2/\sqrt{\delta/2})}}$. In summary, we have 
  \begin{align}
    \forall  \; 1 \le i< j \le n, 
     \quad \forall  \; 1 \le i^\prime< j^\prime\le n, &
      \quad \text{with} \quad (i,j) \neq (i^\prime, j^\prime),\notag \\
    \sE{\aij}  =  \pij
     + \O{\frac{1}{\sqrt{N}}}\mspace{-4mu}, 
      \; & \text{and} \;
      \sE{\roij} = \pijp + \O{\frac{1}{\sqrt{N}}},
      \label{approximations}
  \end{align}
  with probability $1 - \delta/4$. We are now in position to substitute $\sE{\aij}$ and $\sE{\roij}$ with the expressions given
  by (\ref{approximations}), in $\F[2]{\bB}$ given by (\ref{F2_sample}) in lemma~\ref{lemma4}. Using (\ref{approximations}), the
  first term in (\ref{F2_sample}) becomes   
\begin{align}
    \bigg[
      \sum_{(i,j) \in \cE(\bB)}
      \left (1 - 2\sE{a_{ij}}\right)
      + \mspace{-16mu}
      \sum_{1 \le i < j \le n} \sE{a_{ij}}
    \bigg]^2 & \notag\\
    = \left[
      \sum_{(i,j) \in \cE(\bB)}  \mspace{-8mu}
      \left (1 - 2\pij\right) + \mspace{-8mu}
      \sum_{1 \le i < j \le n} \pij
    \right]^2 &
    + \O{\frac{1}{\sqrt{N}}}. \label{mainterm}
  \end{align}
  with probability $1 - \delta/4$.
  Also,   we have 
  \begin{equation}
    \sum_{1 \le i < j \le n} \sE{\aij} \left(1 - \sE{\aij} \right)
    =
    \sum_{1 \le i < j \le n} \pij \left(1 - \pij \right)
    +
    \O{\frac{1}{\sqrt{N}}}. \label{residual}
  \end{equation}
  with probability $1 - \delta/4$. The last two terms in (\ref{F2_sample}) can be neglected since,
  \begin{align}
    \sum_{1 \le i < j \le n}
     & \sum_{\stackrel{1 \le i^\prime < j^\prime \le n }{(i,j) \neq (i^\prime,j^\prime)}}
       \Big [ \sE{\aij}  \sE{\aijp}   - \sE{\roij} \Big] \notag \\
    & = \sum_{1 \le i < j \le n}
      \sum_{\stackrel{1 \le i^\prime < j^\prime \le n }{(i,j) \neq (i^\prime,j^\prime)}}
      \mspace{-12mu} \Big [\pij  \pijp - \pij \pijp \Big] + \O{\frac{1}{\sqrt{N}}}
      = \O{\frac{1}{\sqrt{N}}}, \label{neglect1}
  \end{align}
  with probability $1 - \delta/4$. Similarly
  \begin{equation}
    \sum_{(i,j) \in \cE(\bB)}
    \sum_{(i^\prime,j^\prime) \in \nE(\bB)}
    \sE{\aij} \sE{\aijp} - \sE{\roij} = \O{\frac{1}{\sqrt{N}}} \label{neglect2},
  \end{equation}
  with probability $1 - \delta/4$. Substituting (\ref{mainterm}), (\ref{residual}), (\ref{neglect1}), and (\ref{neglect2}) into
  (\ref{F2_sample}) yields the following estimate
  \begin{equation*}
    \F[2]{\bB} =
    \bigg[
      \sum_{(i,j) \in \cE(\bB)}  \mspace{-16mu}
      \left (1 - 2\pij \right) + \mspace{-16mu} \sum_{1 \le i < j \le n} \mspace{-8mu} \pij
    \bigg]^2
    \mspace{-8mu}  + \mspace{-12mu}
    \sum_{1 \le i < j \le n} \mspace{-16mu} \pij \left(1 - \pij \right)
    + \O{\frac{1}{\sqrt{N}}},
  \end{equation*}
  which holds with   with probability $1 - \delta$. 
\end{proof}
% ________________________________________________________________________  
\subsection{Proof of Theorem~\ref{theorem2}
  \label{proofTheorem2}}
% ________________________________________________________________________
We prove (\ref{sampleFrechetMean}), in theorem~\ref{theorem2}, for the sample
\fr mean. The proof for the sample \fr median is completely similar (it also uses a concentration of measure argument for the
\fr function defined in (\ref{sample-frechet-median})) and is therefore omitted. Because of lemma~\ref{lemma6},
(\ref{F2_concentrates}) implies that
  \begin{equation*}
    \forall \delta \in (0,1),  \exists N_\delta, \forall N\ge N_\delta, \forall \bB \in \cS,\mspace{32mu}
    \F[2]{\bB} = F_2(\bB) + \O{\frac{1}{\sqrt{N}}},
  \end{equation*}
  with probability $1-\delta$ over the realization of the sample $\left\{ \Gk \right\}_{1\leq k \leq N}$. For $N$ large enough,
  the main term dominates the expression of $\F[2]{\bB}$, and we can neglect the $\O{1/\sqrt{N}}$ term. We are left with
  $F_2(\bB)$, the \fr function for the population mean, given by (\ref{frechetMeanOpt}), in lemma~\ref{lemma1}. The minimum of
  $\F[2]{\bB}$ is thus achieved for the adjacency matrix given by the population \fr mean, $\fm{\pr}$, defined by 
  (\ref{populationFrechetMean}), as advertised in (\ref{sampleFrechetMean}), in theorem~\ref{theorem2}. \qed
% ________________________________________________________________________
%\bibliographystyle{amsplain}
  %\bibliography{/Users/francois/LaTeX/Bib/biblio}

  \providecommand{\bysame}{\leavevmode\hbox to3em{\hrulefill}\thinspace}
\providecommand{\MR}{\relax\ifhmode\unskip\space\fi MR }
\providecommand{\MRhref}[2]{%
  \href{http://www.ams.org/mathscinet-getitem?mr=#1}{#2}
}
\providecommand{\href}[2]{#2}

% ________________________________________________________________________
\end{document}